\documentclass[12pt,reqno]{amsart}

\usepackage{amsmath}
\usepackage{amsthm}
\usepackage{amsfonts}
\usepackage{amssymb}
\usepackage{latexsym}
\usepackage{amscd}
\usepackage{stmaryrd}
\usepackage{amsbsy}

\usepackage{txfonts}

\allowdisplaybreaks
\newcommand{\nn}{\nonumber}

\newcommand{\C}{{\mathbb C}}
\newcommand{\Z}{{\mathbb Z}}

\newcommand{\tr}{{\rm tr}}

\newcommand{\ka}{\kappa}

\def\g{\operatorname{\mathfrak{g}}}
\def\tr{\operatorname{tr}}

\numberwithin{equation}{section}

\theoremstyle{plain}
\newtheorem{thm}{Theorem}[section]

\newtheorem{lem}[thm]{Lemma}
\newtheorem{prop}[thm]{Proposition}

\newtheorem{dfn}[thm]{Definition}
\newtheorem{exmp}[thm]{Example}
\newtheorem{re}[thm]{Remark}

\begin{document}
\title{Integral formulas for 
  quantum isomonodromic systems 
}
\date{}

\author{Hajime Nagoya} 
\address{Department of Mathematics, Kobe University,
 Kobe 657-8501, Japan, 
 Research Fellow of the Japan Society for the Promotion of Science}
\email{nagoya@math.kobe-u.ac.jp}

\begin{abstract}
We conisder time-dependent Schr\"odinger systems,  
which are quantizations of the Hamiltonian systems obtained from a similarity reduction of the 
Drinfeld-Sokolov hierarchy  by K.~Fuji and T.~Suzuki, and a similarity reduction of the UC hierarchy by 
T.~Tsuda, independently. These Hamiltonian systems describe isomonodromic deformations for 
certain Fuchsian systems. Thus, our Schr\"odinger systems can be regarded as quantum 
isomonodromic systems. 
Y.~Yamada conjectured that our quantum isomonodromic systems determine  
instanton partition functions in $\mathcal{N}=2$ $SU(L)$ gauge theory. 

The main purpose of this paper is to present  integral formulas as
particular solutions to our quantum isomonodromic systems. 
These integral formulas are  generalizations of the generalized hypergeometric 
function $_LF_{L-1}$. 
\end{abstract}
\maketitle



\section{Introduction}
Fix integers $L\ge 2$ and $N\ge 1$. 
We consider following time-dependent Schr\"odinger system
\begin{equation}\label{eq Sch Intro}
\ka \frac{\partial}{\partial z_i}\Psi({\bf q,z})=H_i\left({\bf q, \frac{\partial}{\partial q},z}\right)\Psi({\bf q,z})\quad 
(1\le i\le N)
\end{equation}
where $\ka\in \C$ and $\Psi({\bf q}, {\bf z})$ is an unknown function of 
\begin{equation*}
{\bf q}=\left(q_1^{(1)},\ldots, q_{L-1}^{(1)}, q_1^{(2)},\ldots, q_{L-1}^{(2)},\ldots, q_1^{(N)},\ldots, q_{L-1}^{(N)}\right)
\end{equation*}
and ${\bf z}=(z_1,\ldots,z_N)$. The Hamiltonians $H_i$ are defined in Definition \ref{def Hamiltonian}.

The Schr\"odinger system \eqref{eq Sch Intro} is a quantization of the classical Hamiltonian system $\mathcal{H}_{L,N}$  obtained from a similarity reduction of the Drinfeld-Sokolov hierarchy by 
K.~Fuji and T.~Suzuki ($L=3, N=1$) \cite{FS}, T.~Suzuki ($L\ge 2, N=1$) \cite{S1}, and  
a similarity reduction of the UC hierarchy by 
T.~Tsuda 
($L\ge 2, N\ge 1$) \cite{T2}, independently.  In \cite{T2}, 
T.~Tsuda 
showed that the classical Hamiltonian system $\mathcal{H}_{L,N}$ is 
equivalent to a Schlesinger system governing 
isomonodromic deformation for a certain Fuchsian system. 

On the other hand, 
Y.~Yamada conjectured  in the context of the so-called AGT relation that 
the instanton partition function, in the presence of the full surface operator in $\mathcal{N} = 2$ $SU(L)$ gauge theory, is determined by
the Schr\"odinger system \eqref{eq Sch Intro}  for $N=1$ \cite{Y}.   
In the case of $L=2$, the Schr\"odinger system \eqref{eq Sch Intro} is a quantization of the Garnier systems \cite{Garnier}, 
\cite{KO}, which has been appeared in the conformal field theory  \cite{Teschner}. 

In this paper, we present a family of hypergeometric integrals as particular solutions 
to the Sch\"rodinger system \eqref{eq Sch Intro}. These solutions are polynomials in ${\bf q}$
with the degree $M\in \Z_{\ge 1}$ and  the coefficients are 
integral representations of hypergeometric type. 

A key to find special solutions to quantum isomonodromic systems is to observe 
special solutions to the corresponding classical isomonodromic systems. For example, 
both the classical and  quantum sixth Painlev\'e equation has 
a particular solution expressed 
in terms of 
the Gauss hypergeometric function  \cite{N HGS}. 

It is known that 
the classical Hamiltonian system $\mathcal{H}_{L,N}$ has a particular solution  
expressed in terms of a generalization of the Gauss hypergeometric function by 
T.~Suzuki ($L\ge 2, N=1$) \cite{S2}, T.~Tsuda 
($L\ge 2, N\ge 1$) 
\cite{T1}. 
Observing the linear Pfaffian system derived from this generalization of the Gauss hypergeometric function, we see  
indeed that  hypergeometric integrals given in \cite{T1} yield a particular 
solution to the Schr\"odinger system \eqref{eq Sch Intro}:
\begin{thm}
The integral formula
\begin{equation}\label{eq IF Intro M=1}
\int_\Delta\prod_{n=1}^{L-1}t_n^{\alpha_n/\ka}
\prod_{i=1}^N\left(
1-z_it_{L-1}
\right)^{-\beta_i/\ka}\prod_{n=1}^{L-1}\left(
t_{n-1}-t_n
\right)^{-\gamma_n/\ka}
\left(\varphi_0(t)-\sum_{i=1}^N\sum_{n=1}^{L-1}\varphi_n^{(i)}(t)q_n^{(i)}\right), 
\end{equation}
which is a polynomial in ${\bf  q}$ with the degree $1$, is a particular solution to the 
Schr\"odinger system \eqref{eq Sch Intro}. 
Here $\Delta$ is a twist cycle and $\varphi_0(t)$, $\varphi_n^{(i)}(t)$ are certain 
rational ($L-1$)-forms defined in \eqref{eq basis phi}. 
\end{thm}
(see Theorem \ref{thm M=1})

\medskip 

In order to generalize  the integral formula as particular solution to the case of 
polynomials in ${\bf q}$ with the degree $M\in\Z_{\ge 2}$, let us recall equivalence between 
the Knizhnik-Zamolodchikov equation of 
the 
conformal field theory and a quantization of a Schlesinger system \cite{Har}, \cite{R}. 
The KZ equations for the simple Lie algebra $\g$, have integral representations as 
solutions taking values in tensor products of Verma modules of $\g$  (see, for example, \cite{ATY}, 
\cite{SV}). 
From the point of view that the integral formula \eqref{eq IF Intro M=1} may be a solution to 
the Knizhnik-Zamolodchikov equation, it should be viewed that the integral variables are corresponding to the simple roots  of 
$\frak{sl}_L$. While, for the case of $L=2$ and $N=1$, it is known that the Schr\"odinger system \eqref{eq Sch Intro}, the quantum sixth 
Painlev\'e equation, has hypergeometric solutions \cite{N HGS}:
\begin{equation*}
\int_\Delta \prod_{1\le a<b\le M}(t^{(a)}-t^{(b)})^{2/\ka}\prod_{a=1}^M(t^{(a)})^{\alpha/\ka}(1-zt^{(a)})^{-\beta/\ka}
(1-t^{(a)})^{-\gamma/\ka}\left(
\varphi_0(t^{(a)})-\varphi_1^{(1)}(t^{(a)})q_1^{(1)}
\right).
\end{equation*}
Note that the integrand above consists of $M$-copies of the integrand of \eqref{eq IF Intro M=1} multiplied by the 
coupled term $\prod_{1\le a<b\le M}(t^{(a)}-t^{(b)})^{2/\ka}$.

Considering upon these, we arrive at 
\begin{thm}
The integral formula
\begin{align*}
\int_\Delta& \prod_{1\le a<b\le M, \atop 1\le n\le L-1}\left(t_n^{(a)}-t_n^{(b)}\right)^{2/\kappa}
\prod_{ 1\le a, b\le M,\atop 1\le n\le L-2 }
\left(t_n^{(a)}-t_{n+1}^{(b)}\right)^{-1/\kappa}
\\
&\times \prod_{a=1}^M\left\{\prod_{n=1}^{L-1}\left(
t_n^{(a)}
\right)^{\alpha_n/\ka}
\prod_{i=1}^N\left(1-z_it^{(a)}_{L-1}\right)^{-\beta_i/\ka}\left(
1-t_1^{(a)}
\right)^{-\gamma/\ka}
\left(\varphi_0(t^{(a)})-\sum_{i=1}^N\sum_{n=1}^{L-1}\varphi_n^{(i)}(t^{(a)})q_n^{(i)}\right)
\right\},
\end{align*}
which is a polynomial in ${\bf q}$ with the degree $M$, is a particular solution to the Schr\"odinger system 
\eqref{eq Sch Intro}. Here $\Delta$ is a skew-symmetric twist cycle. 
\end{thm}
(see Theorem \ref{thm IF})

\medskip

The remainder of this paper is organized as follows. 
In section 2, we introduce  quantizations of the classical Hamiltonians of 
$\mathcal{H}_{L,N}$ and show that those quantum Hamiltonians are  mutually 
commutative. In section 3, we introduce our Schr\"odinger systems and discuss  properties of 
them. In section 4, 
we give integral formulas for solutions. 

\begin{re} As mentioned above, 
the classical Hamiltonian system $\mathcal{H}_{L,N}$ describes isomonodromic 
deformation for an $L\times L$ Fuchsian system
\begin{equation*}
\frac{\partial}{\partial u}\Phi(u) =\sum_{i=0}^{N+1}\frac{A_i}{u-u_i}\Phi(u),
\end{equation*}
where $u_0=1$, $u_i=1/z_i$ ($1\le i\le N$), and $u_{N+1}=0$, 
whose spectral type is given by the ($N+3$)-tuple
\begin{equation*}
(1,1,\ldots,1),(1,1,\ldots,1),(L-1,1),\ldots,(L-1,1)
\end{equation*}
of partitions of $L$. A spectral type defines multiplicities of the eigenvalues of 
each residue matrix $A_i$. 
Consequently, $L-1$  parameters are associated with 
 singular points $0$ and $\infty$, and one parameter is associated with 
 each singular point $u_i$ for $i=0,\ldots,  N$. Notice that in the integrand given in the Theorems above, $L-1$ parameters are associated with 
the singular point $0$, and one parameter is associated with each singular point 
$1$, $1/z_i$ ($1\le i\le N$). 
\end{re}

\section{Hamiltonian}
Let us define a non-commutative associative algebra $W_{L,N}$ over $\C$ with generators
\begin{align*}
&q_m^{(i)},  p_m^{(i)} \quad (1\le m\le L-1, \ 1\le i\le N),  
\\
&e_n,  \ka_n,  \theta_j, \hbar \quad (0\le n\le L-1, \ 0\le j\le N)
\end{align*}
and  commutation relations 
\begin{equation}\label{eq com rel}
\left[
 p_m^{(j)}, q_n^{(i)}
\right]=\delta_{n,m}\delta_{i,j}\hbar\quad (1\le n, m\le L-1,\ 1\le i, j\le N), 
\end{equation}
where $\delta_{i,j}$ is Kronecker's delta, 
and the other commutation relations are zero, and relations
\begin{equation*}
\sum_{m=0}^{L-1}e_{m}=\frac{L-1}{2},\quad \sum_{m=0}^{L-1}\ka_{m}=\sum_{i=0}^{N}\theta_{i}. 
\end{equation*} 

The non-commutative associative algebra $W_{L,N}$  is an Ore domain, so that 
we can define its skew field $\mathcal{K}_{L,N}$ (see, for example, \cite{bjork}, Chapter 1, Section 8). 

\medskip

\begin{dfn}\label{def Hamiltonian}
\begin{em}
We introduce Hamiltonians $H_i$ ($i=1,\ldots, N$) in the rational function field $W_{L,N}(z_{1},\ldots, z_{N})$ 
in variables $z_{1}$, \ldots, $z_{N}$ 
by 
\begin{align}
z_iH_i
=&\sum_{n=0}^{L-1}e_nq_n^{(i)}p_n^{(i)}
+\sum_{j=0}^N\sum_{0\le m<n\le L-1}q_m^{(i)}p_m^{(j)}q_n^{(j)}p_n^{(i)}
+\frac{1}{z_i-1}\sum_{m,n=0}^{L-1}q_m^{(i)}p_m^{(0)}q_n^{(0)}p_n^{(i)} \nn
\\
&+\sum_{j=1\atop  j\neq i}^N\frac{z_j}{z_i-z_j}\sum_{m,n=0}^{L-1}
q_m^{(i)}p_n^{(i)}q_n^{(j)}p_m^{(j)}
+\theta_i\left(
e_0+\ka_0-\sum_{j=1}^N\theta_j-\sum_{j=1,\atop j\neq i}^N\frac{\theta_j z_j}{z_i-z_j}
\right)
, \label{eq def Hamiltonian}
\end{align}
where 
\begin{align*}
&q_0^{(i)}=\theta_i+\sum_{m=1}^{L-1}q_m^{(i)}p_m^{(i)},\quad p_0^{(i)}=-1 \quad (1\le i\le N), 
\\
&q_m^{(0)}=-1,\quad p_m^{(0)}=\kappa_m+\sum_{i=1}^Nq_m^{(i)}p_m^{(i)},\quad (1\le m\le L-1),
\\
& q_0^{(0)}=\ka_{0}-\sum_{i=1}^{N}\theta_{i}-\sum_{i=1}^{N}\sum_{m=1}^{L-1}q_m^{(i)}p_m^{(i)},\quad p_0^{(0)}=-1.  
\end{align*}
\end{em}
\end{dfn}

The Hamiltonians $H_i$ ($i=1,\ldots, N$) are canonical quantization of the polynomial Hamiltonians in \cite{T2}, Appendix A.  What we mean by canonical quantization 
is, to replace the Poisson bracket with the commutator. 

Since the canonical variables  in the classical Hamiltonians are not 
separated,  
quantization of the Hamiltonians is not 
unique. In the following, we show that the Hamiltonians $H_i$ are mutually commutative and the Schr\"odinger equations 
associated with the Hamiltonians $H_i$ have integral formulas. 

\begin{exmp}
We give an example of the Hamiltonians $H_{i}$ in the case of $L=2$.  
   Set $\left(q_{i}, p_{i}\right)=\left(q_{1}^{(i)}, p_{1}^{(i)}\right)$.  
The Hamiltonian $H_{i}$ is expressed as follows:
\begin{align*}
z_{i}(z_{i}-1)H_{i}=&q_{i}\left(
\ka_{1}-\theta_{0}+\sum_{j=1}^{N}q_{j}p_{j}
\right)\left(
\ka_{1}+\sum_{j=1}^{N}q_{j}p_{j}
\right)+z_{i}\left(
\theta_{i}+q_{i}p_{i}
\right)p_{i}
\\
&-\sum_{j=1,\atop j\neq i}^{N}\frac{z_{j}}{z_{i}-z_{j}}\left(
\theta_{j}+q_{j}p_{j}
\right)q_{i}p_{j}
-\sum_{j=1,\atop j\neq i}^{N}\frac{z_{i}}{z_{i}-z_{j}}\left(
\theta_{i}+q_{i}p_{i}
\right)q_{j}p_{i}
\\
&-\sum_{j=1,\atop j\neq i}^{N}\frac{z_{i}(z_{j}-1)}{z_{j}-z_{i}}\left(
\theta_{i}+q_{i}p_{i}
\right)q_{j}p_{j}
-\sum_{j=1,\atop j\neq i}^{N}\frac{z_{i}(z_{j}-1)}{z_{j}-z_{i}}\left(
\theta_{j}+q_{j}p_{j}
\right)q_{i}p_{i}
\\
&-(z_{i}+1)\left(
\theta_{i}+q_{i}p_{i}
\right)q_{i}p_{i}-
\left(
(e_{1}-e_{0})z_{i}+e_{0}-e_{1}-\hbar+\ka_{1}-\ka_{0}
\right)q_{i}p_{i}
\end{align*}
plus some function in only $(z_{1},\ldots, z_{N})$. These Hamiltonians are  quantizations 
of the polynomial Hamiltonians for the Garnier system  \cite{KO}. 
\end{exmp}

\begin{exmp}\label{ex FS}
We give an example of the Hamiltonians $H_1$ in the case of $N=1$. 
 Set $\left(q_{m}, p_{m}\right)=\left(q_{m}^{(1)}, p_{m}^{(1)}\right)$, $H=H_1$, and $z=z_1$.  
The Hamiltonian $H$ is written in a {\it coupled} form  as follows: 
\begin{align*}
z(z-1)H=&\sum_{m=1}^{L-1}H_{\mathrm{VI}}\left(
\sum_{n=0}^{L-1}\alpha_{2n+1}-\alpha_{2m-1}-\eta, \sum_{n=0}^{m-1}\alpha_{2n}, \sum_{n=m}^{L-1}\alpha_{2n}, 
\alpha_{2n-1}\eta; q_m, p_m
\right)
\\
&+\frac{1}{4}\sum_{1\le m< n\le L-1}\left(
\left(
\left(
q_m-1
\right)p_mq_m+q_mp_m\left( q_m-1\right)+2\alpha_{2m-1}\left(q_m-1\right)
\right)
\left(
p_n\left(
q_n-z
\right)+\left(
q_n-z
\right)p_n
\right)\right.
\\
&+\left.
\left(
\left(
q_n-z
\right)p_nq_n+q_np_n\left(
q_n-z
\right)
+2\alpha_{2n-1}\left(q_n-z\right)
\right)
\left(
p_m\left(
q_m-1
\right)
+\left(
q_m-1
\right)p_m  
\right)
\right)
\end{align*}
plus some function in only $(z_{1},\ldots, z_{N})$, 
where 
\begin{align*}
H_{\mathrm{VI}}\left(
a_0, a_1, a_z, a; q,p
\right)=&
\frac{1}{6}\left(qp(q-1)p(q-z)
+(q-1)p(q-z)pq
+(q-z)pqp(q-1)+\right.
\\
&\left.+(q-z)p(q-1)pq
+(q-1)pqp(q-z)
+qp(q-z)p(q-1)\right)\nonumber
\\
&-\frac{1}{2}\left(a_0((q-1)p(q-z)
+(q-z)p(q-1))
+a_1(qp(q-z)+(q-z)pq)
\right.\nonumber
\\
&\left.+
(a_z-1)(qp(q-1)
+(q-1)pq)
\right)+a q. 
\nonumber
\end{align*}
Here, we let   
\begin{align*}
&\alpha_{2m-1}=\ka_n-\hbar\quad (1\le m\le L-1), \quad 
\alpha_{2m}=e_{m}-e_{m+1}-\ka_m+\hbar \quad (1\le m\le L-2),
\\
&\alpha_0=e_0-e_1,\quad \alpha_{2L-1}=-\ka_0+(L-2)\hbar,\quad \sum_{m=0}^{2L-1}\alpha_m=\kappa,
\quad \eta=-\ka_0+\theta_1-\frac{L-2}{2}\hbar. 
\end{align*}
The Hamiltonian $H$ is a quantization of the Hamiltonian obtained by Fuji-Suzuki ($L=3$) \cite{FS} and 
Suzuki ($L\ge 3$) \cite{S1}. The Hamiltonian $H_{\mathrm{VI}}$ is of the sixth quantum Painlev\'e equation with the affine Weyl group symmetry of type $D_4^{(1)}$ 
introduced in \cite{N quantum PVI}. 
\end{exmp}

\subsection{Commutativity}
The Hamiltonians $H_{i}$ ($i=1,\ldots, N$) are expressed in the following forms 
\begin{equation*}
-z_{i}^{2}H_{i}=\sum_{j=0,\atop j\neq i}^{N+1}\frac{\Omega_{i,j}}{u_{i}-u_{j}}, 
\end{equation*}
where $\Omega_{i,j}$ are elements in $W_{L,N}$ and $u_0=1$, $u_{i}=1/z_{i}$ ($i= 1,\ldots, n$) and $u_{N+1}=0$. 

For $i, j=1,\ldots, N$,  the forms $\Omega_{i,j}$ read as 
\begin{equation}\label{eq Omega}
\Omega_{i,j}=\frac{1}{2}\tr \left(
\widehat{A}^{(i)}\widehat{A}^{(j)}
\right), 
\end{equation}
where $\widehat{A}^{(i)}$ is a $L\times L$ matrix defined as 
\begin{equation}\label{eq A qp}
\left(\widehat{A}^{(i)}\right)_{m,n}=q_{m}^{(i)}p_{n}^{(i)}
\end{equation}
for $m,n=0,1,\ldots, L-1$, where $\left(\widehat{A}^{(i)}\right)_{m,n}$ is the ($m,n$) entry of the matrix $\widehat{A}^{(i)}$.   
%
The entries of $\widehat{A}^{(i)}$  satisfy the following commutation relations. 
\begin{lem}\label{lem A com}
It holds that 
\begin{equation}\label{eq A com}
\frac{1}{\hbar}
\left[\left(\widehat{A}^{i}\right)_{m,n}, \left(\widehat{A}^{j}\right)_{m',n'}\right]=\delta_{i,j}\left(
\delta_{n,m'}\left(\widehat{A}^{i}\right)_{m,n'}-\delta_{n',m}\left(\widehat{A}^{i}\right)_{m',n}\right)
\end{equation}
for $0\le m, n, m', n'\le L-1$ and $1\le i,j\le N$. 
\end{lem}
A proof is given by a straightforward calculation. 

Recall the definition of the Gaudin Hamiltonians  (see, for example, \cite{Har}, Section 2). 
The Gaudin Hamiltonians $G_i$ ($i=1,\ldots, N$) for $\frak{gl}_L$ are defined as 
\begin{equation*}
G_i=\frac{1}{2}\sum_{j=1,\atop j\neq i}^{N}\frac{\tr \left(
{B}^{(i)}{B}^{(j)}
\right)}{u_{i}-u_{j}}, 
\end{equation*}
where  $B^{(i)}$ ($i=1,\ldots, N$) are $L\times L$ matrices whose entries satisfy the commutation relations \eqref{eq A com}. 
Since the commutativity of Gaudin Hamiltonians equals  the so-called {\it infinitesimal braid relations}, 
 Lemma \ref{lem A com} yields:
\begin{align}
&\left[\Omega_{i,j}, \Omega_{k,l} \right]=0\quad (i,j,k,l \ \text{are distinct}),
\\
&\left[\Omega_{i,j}, \Omega_{i,k}+\Omega_{k,j}\right]=0\quad (i,j,k \ \text{are distinct})
\end{align}
for $i, j,k,l=1,\ldots, N$.  

The other elements $\Omega_{i, 0}$ and $\Omega_{i,N+1}$ ($i=1,\ldots, N$) are not expressed in a similar way as 
\eqref{eq Omega} and \eqref{eq A qp}. However, we can check 
by a straightforward calculation 
that the {\it infinitesimal braid relations} above hold even if $i,j,k,l=0,1,\ldots, N+1$. 
 Therefore, we have 
\begin{prop}\label{prop commutativity}
Hamiltonians $H_{i}$ ($i=1,\ldots, N$) are mutually commutative.  
\end{prop}


\section{Schr\"odinger system}

 Denote by 
\begin{equation}\label{eq hamiltonian sch}
H_{i}\left(
{\bf q, \frac{\partial}{\partial  q},z}
\right)
\end{equation}
for $i=1,\ldots, N$,  the Hamiltonians obtained by substituting $q_m^{(i)}$ and $\partial /\partial q_m^{(i)}$ into 
$q_m^{(i)}$ and $p_m^{(i)}$, resepectively, 
of the Hamiltonians  $H_i$ defined in Definition \ref{def Hamiltonian}. 

We consider the following Schr\"odinger system: 
\begin{equation}\label{eq Schrodinger}
\kappa \frac{\partial }{\partial z_{i}}\Psi({\bf q}, {\bf z})=H_{i}\left(
{\bf q}, \frac{\partial}{\partial {\bf q}}, {\bf z}
\right)\Psi({\bf q}, {\bf z})\quad
\end{equation}
where $\ka\in \C$, $\Psi({\bf q}, {\bf z})$ is an unknown function of 
\begin{equation*}
{\bf q}=\left(q_1^{(1)},\ldots, q_{L-1}^{(1)}, q_1^{(2)},\ldots, q_{L-1}^{(2)},\ldots, q_1^{(N)},\ldots, q_{L-1}^{(N)}\right)
\end{equation*}
and ${\bf z}=(z_1,\ldots,z_N)$. Here, we regard $e_n$, $\ka_n$, $\theta_i$ as complex parameters.

\begin{prop}
The Schr\"odinger system \eqref{eq Schrodinger} is completely integrable in the sense of Frobenius, that is, 
it holds 
\begin{equation*}
\left[
\kappa \frac{\partial }{\partial z_{i}}-H_{i}\left(
{\bf q}, \frac{\partial}{\partial {\bf q}}, {\bf z}
\right),
\kappa \frac{\partial }{\partial z_{j}}-H_{j}\left(
{\bf q}, \frac{\partial}{\partial {\bf q}}, {\bf z}
\right)
\right]=0,  
\end{equation*}
 for $i,j=1,\ldots, N$. 
\end{prop}
\begin{proof}
Thanks to Proposition \ref{prop commutativity}, we have only to show 
\begin{equation*}
\frac{\partial }{\partial z_{i}}H_{j}\left(
{\bf q}, \frac{\partial}{\partial {\bf q}}, {\bf z}
\right)=\frac{\partial }{\partial z_{j}}H_{i}\left(
{\bf q}, \frac{\partial}{\partial {\bf q}}, {\bf z}
\right). 
\end{equation*}

It is easily calculated as follows. For $i\neq j$, we have 
\begin{align*}
\frac{\partial }{\partial z_{i}}H_{j}\left(
{\bf q}, \frac{\partial}{\partial {\bf q}}, {\bf z}
\right)=&\frac{\partial }{\partial z_{i}}\left(
\frac{z_i}{z_j(z_j-z_i)}\left(\sum_{m,n=0}^{L-1}
q_m^{(j)}p_n^{(j)}q_n^{(i)}p_m^{(i)}-\theta_i\theta_j\right)
\right)
\\
=&\frac{1}{(z_i-z_j)^2}\left(\sum_{m,n=0}^{L-1}
q_m^{(j)}p_n^{(j)}q_n^{(i)}p_m^{(i)}-\theta_i\theta_j\right). 
\end{align*}
 From Lemma \ref{lem A com}, the last line is symmetrical with respect to $i$ and $j$. 
Thus, we finished the proof. 
\end{proof}

In the most simplest case, namely, the case of $L=2$ and $N=1$, the Schr\"odinger system \eqref{eq Schrodinger}
is the quantum sixth Painlev\'e equation. In the previous work \cite{N HGS}, we showed that the quantum 
sixth Painlev\'e equation has  polynomial solutions in terms of $q$.

In the general case, the Schr\"odinger system \eqref{eq Schrodinger} also has polynomial solutions 
in terms of ${\bf q}$ due to the following proposition.

For a  $L-1\times N$ matrix $A$ whose entries are non-negative integers, 
let $q^A$ be the monomial defined by 
\begin{equation*}
q^A=\prod_{m=1}^{L-1}\prod_{i=1}^N \left(
q_m^{(i)}
\right)^{A_{m,i}}, 
\end{equation*}
where $A_{m,i}$ is the ($m,i$) entry of the matrix $A$. Set $d(A)=\sum_{m=1}^{L-1}\sum_{i=1}^N A_{m,i}$. 

Let $V(M)$ ($M\in \Z_{\ge 0}$) be the subspace of the polynomial ring $\C[{\bf q}]$ such that 
the degree of elements in $V(M)$ is less than $M$, namely, $V(M)=\bigoplus_{A}\C q^A$, where 
the summation is taken over all $L-1\times N$ matrices $A$ such that $d(A)\le M$.

\begin{prop}\label{prop action H}
For each $i=1,\ldots, N$, 
the Hamiltonian $H_i({\bf q}, \partial /\partial {\bf q}, {\bf z})$  acts on $V(M)$ if 
$\ka_0-\sum_{i=1}^N\theta_i=M$. 
\end{prop}
\begin{proof}
We compute the action of the Hamiltonian $H_i({\bf q}, \partial /\partial {\bf q}, {\bf z})$  on $q^A$ such that $d(A)=M$ 
as follows. 
\begin{align}
z_i(z_i-1)H_{i}\left(
{\bf q}, \frac{\partial}{\partial {\bf q}}, {\bf z}
\right) q^A=&-\sum_{n=1}^{L-1}q_n^{(i)}q_0^{(0)}p_n^{(0)} q^A+f({\bf q})
\nonumber
\\
=&-\sum_{n=1}^{L-1}\left(
\ka_0-\sum_{i=1}^N\theta_i-M
\right)\left(
\ka_n+\sum_{i=1}^N A_{n,i}
\right)q_n^{(i)}q^A+f({\bf q}). \label{eq action H}
\end{align}
Here $f({\bf q})$ is a polynomial whose degree is equal to or less than $M$.  
  Hence, 
if $\ka_0-\sum_{i=1}^N\theta_i=M$, then the first term of \eqref{eq action H} vanishes, which finishes the proof. 
\end{proof}

By virtue of Proposition \ref{prop action H}, for the Schr\"odinger equation \eqref{eq Schrodinger}, 
we can consider polynomial solutions 
\begin{equation*}
\Psi({\bf q}, {\bf z})=\sum_{A\in\mathcal{A}_M}c_A({\bf z}) q^A, 
\end{equation*}
where 
\begin{equation}\label{def matrix A}
\mathcal{A}_M=\left\{
A=\left(A_{m,i}\right)\big| A_{m,i}\in\mathbb{Z}_{\ge 0},\ d(A)\le M
\right\}
\end{equation}
  and $c_A({\bf z})$ is a function of ${\bf z}$.  
 In the next section, we present integral formulas taking values in $V(M)$ and show that 
 they are solutions to the Schr\"odinger system \eqref{eq Schrodinger}.

 The Hamiltonians $H_i$ act on another subspaces of the polynomial ring $\C[{\bf q}]$. 
Let $F( T_{1},\ldots,T_{L-1})$ ($T_{1},\ldots, T_{L-1} \in\Z_{\ge 0}$) be the subspace 
of the polynomial ring $\C[{\bf q}]$ 
defined as 
$
F(T_1,\ldots, T_{L-1})=\bigoplus_{A}\C q^A, 
$
where the summation is taken over all $L-1\times N$ matrices $A$ such that the entries of $A$ are 
non-negative integers  and 
$\sum_{i=1}^N A_{m,i}\le T_m$  ($m=1,\ldots, L-1$). Set $d_m(A)=\sum_{i=1}^N A_{m,i}$. 

\begin{prop}
For each $i=1,\ldots, N$, 
the Hamiltonian $H_i({\bf q}, \partial /\partial {\bf q}, {\bf z})$  acts on $F(T_1,\ldots, T_{L-1})$ if 
$\ka_m=-T_m$. 
\end{prop}
\begin{proof} 
Take a $L-1\times N$ matrix $A$ such that the entries of $A$ are 
non-negative integers  and $d_m(A)=T_m$ for any $m\in\{1,\ldots, L-1\}$. 
The Hamiltonian $H_i({\bf q}, \partial /\partial {\bf q}, {\bf z})$  acts on $q^A$ 
as follows:
\begin{align}
z_i(z_i-1)H_{i}\left(
{\bf q}, \frac{\partial}{\partial {\bf q}}, {\bf z}
\right) q^A=&-\left(q_m^{(i)}q_0^{(0)}p_m^{(0)} 
+\sum_{n=m+1}^{L-1}q_m^{(i)}p_m^{(0)}p_n^{(i)} 
+\sum_{n=1,\atop n\neq m}^{L-1}q_m^{(i)}p_m^{(0)}p_n^{(i)}\right)q^A
+f({\bf q})
\nonumber
\\
=&-\left(
\ka_m+T_m
\right)
\left\{
\ka_0-\sum_{j=1}^N\theta_j-d(A)
+\sum_{n=m+1}^{L-1}p_n^{(i)} 
+\sum_{n=1,\atop n\neq m}^{L-1}p_n^{(i)} \right\}q_m^{(i)}q^A
+f({\bf q}). \label{eq action H 2}
\end{align}
Here, $f({\bf q})$ is a polynomial such that  as a polynomial in terms of $q_m^{(1)},\ldots, q_m^{(N)}$, 
the degree of $f({\bf q})$ is equal to or less than $d_m(A)$. Thus, if $\ka_m=-T_m$, then  
the first term of \eqref{eq action H 2} vanishes, which finishes the proof. 
 \end{proof}

Consequently, we can also consider polynomial solutions taking values in $F(T_1,\ldots, T_{L-1})$.


\section{Integral formula}

In this section, we construct integral formulas for the Schr\"odinger systems \eqref{eq Schrodinger}, 
as particular solutions. 

Recall that the Gauss hypergeometric function is a particular solution to both the classical and 
quantum sixth Painlev\'e equation \cite{N HGS}. 
Hypergeometric solutions to the classical Hamiltonian systems $\mathcal{H}_{L,N}$  were given by 
T.~Suzuki ($L\ge 2$, $N=1$) \cite{S2} and T.~Tsuda ($L\ge 2$, $N\ge 1$) \cite{T1}   
independently, under the condition $\ka_0-\sum_{i=1}^N\theta_i=0$.   

These hypergeometric solutions are the generalized hypergeometric functions 
(Thomae's hypergeometric function) $_LF_{L-1}$ in the case of ($L\ge 2$, $N=1$) and 
their generalizations in the case of ($L\ge 2$, $N\ge 1$). 

We expect that these generalized hypergeometric functions are also solutions to a quantization 
of the classical Hamiltonian systems $\mathcal{H}_{L,N}$, 
the Schr\"odinger systems \eqref{eq Schrodinger}. 
Indeed, this is true if we consider polynomial solutions to the Schr\"odinger systems \eqref{eq Schrodinger} with 
$   \ka_0-\sum_{i=1}^N\theta_i=1$. 

Set $\ka_0-\sum_{i=1}^N\theta_i=M\in\mathbb{Z}_{\ge 0}$. 
We begin with the case $M=1$ and later we deal with general case.

\subsection{The case of $M=1$}
%

Consider a multivalued function 
\begin{equation*}
U(t)=\prod_{n=1}^{L-1}t_n^{\alpha_n/\ka}
\prod_{i=1}^N\left(
1-z_it_{L-1}
\right)^{-\beta_i/\ka}\prod_{n=1}^{L-1}\left(
t_{n-1}-t_n
\right)^{-\gamma_n/\ka}
\end{equation*}
with $t_0=1$ defined on the complement $T\in\C^{(L-1)}$ of  singular locus $D$ given by 
\begin{equation*}
D=
\bigcup_{ 1\le n\le L-1}\left\{
t_{n-1}=t_{n}
\right\}
\cup
\bigcup_{ 1\le n\le L-1}
\left\{
t_n= 0
\right\}
   \cup
   \bigcup_{ 1\le i\le N}\left\{
    t_{L-1}= 1/z_i
    \right\}
    . 
\end{equation*}

Let $\mathcal{S}$ be the rank one local system  determined by $U(t)$ and 
 $\mathcal{S}^*$, the dual local system of $\mathcal{S}$. 
The hypergeometric paring between the twisted homology group and  twisted de Rham cohomology group 
 is 
\begin{align*}
H_{L-1}(T, \mathcal{S}^*)\times H^{L-1}(T,\nabla)&\longrightarrow \C
\\
\left(
\Delta, \varphi
\right)&\longmapsto \int_\Delta U(t)\varphi, 
\end{align*}
where $\varphi$ is a rational ($L-1$)-form holomorphic outside $D$ and $\nabla$ 
is the covariant differential operator given by 
$\nabla=d+d\log(U(t))\wedge$.

According to 
\cite{T2},   the following rational ($L-1$)-forms
\begin{equation}\label{eq basis phi}
\varphi_0(t)=\frac{dt_1\wedge\cdots\wedge dt_{L-1}}{t_{L-1}}
\prod_{n=1}^{L-1}\frac{1}{t_{n-1}-t_n}, 
\quad \varphi^{(i)}_n(t)=\frac{dt_1\wedge\cdots\wedge dt_{L-1}}{(1-z_it_{L-1})t_{L-1}}\prod_{m=1,\atop m\neq n}^{L-1}
\frac{1}{t_{m-1}-t_m}
\end{equation}
represent a basis of $H^{L-1}(T,\nabla)$.  

Define the integral formula $\Psi_1({\bf q,z})$ by 
\begin{equation*}
\Psi_1({\bf q,z})=\int_\Delta U(t)\left(\varphi_0(t)-\sum_{i=1}^N\sum_{n=1}^{L-1}\varphi_n^{(i)}(t)q_n^{(i)}\right)
\end{equation*}
with $\Delta\in H_{L-1}(T, \mathcal{S}^*)$. From Proposition \ref{prop action H}, when $\ka_0-\sum_{i=1}^N\theta_i=1$, the action of Hamiltonian 
$H_i$ ($i=1,\ldots, N$) on the integral formula, 
$H_i\Psi_1({\bf q,z})$, is also 
a polynomial of degree equal to or less than $1$ and then the constant term and the coefficient of $q_n^{(j)}$ ($1\le n\le L-1$, $1\le j\le N$) of $H_i\Psi_1({\bf q,z})$ 
are 
linear combinations of $\int_\Delta U(t)\varphi_0(t)$ and $\int_\Delta U(t)\varphi_n^{(j)}(t)$ ($1\le n\le L-1$, $1\le j\le N$). 
Remarkably they coincide with $\ka \partial \varphi_0(t)/ \partial z_i$ and $\ka \partial \varphi_n^{(j)}(t)/\partial z_i$ 
with appropriate correspondence between parameters. 
Namely, we have
\begin{thm}\label{thm M=1}
If $\ka_0-\sum_{i=1}^N\theta_i=1$, then the integral formula $\Psi_1({\bf q,z})$ is a solution to the 
Schr\"odinger system \eqref{eq Schrodinger}, with 
\begin{equation*}
\alpha_n=e_{n+1}-e_n+\ka_{n+1},\quad \beta_i=-\theta_i, 
\quad \gamma_n=\ka_n,
\end{equation*}
for $1\le n\le L-1$ and $1\le i\le N$, where $e_L=e_0$ and $\ka_{L}=1$. 
\end{thm}

\subsection{The case of $M\ge 2$}

Fix $M\in\Z_{\ge 2}$. 
We consider a multivalued function
\begin{align*}
U(t)=& \prod_{1\le a<b\le M, \atop 1\le n\le L-1}\left(t_n^{(a)}-t_n^{(b)}\right)^{2/\kappa}
\prod_{ 1\le a, b\le M,\atop 1\le n\le L-2 }
\left(t_n^{(a)}-t_{n+1}^{(b)}\right)^{-1/\kappa}
\\
&\times \prod_{a=1}^M\left\{\prod_{n=1}^{L-1}\left(
t_n^{(a)}
\right)^{\alpha_n/\ka}
\prod_{i=1}^N\left(1-z_it^{(a)}_{L-1}\right)^{-\beta_i/\ka}\left(
1-t_1^{(a)}
\right)^{-\gamma/\ka}\right\}
\end{align*}
 defined on the complement $T\in\C^{(L-1)M}$ of  singular locus $D$ given by 
\begin{equation*}
D=\bigcup_{1\le a<b\le M,\atop 1\le n\le L-1} \left\{t^{(a)}_n= t^{(b)}_n\right\}
\cup
\bigcup_{1\le a, b\le M,\atop 1\le n\le L-2}\left\{
t^{(a)}_n=t^{(b)}_{n+1}
\right\}
\cup
\bigcup_{1\le a\le M,\atop 1\le n\le L-1}
\left\{
t^{(a)}_n= 0
\right\}
   \cup
   \bigcup_{1\le a\le M,\atop 1\le i\le N}\left\{
    t^{(a)}_{L-1}= 1/z_i
    \right\}
    \cup
\bigcup_{1\le a\le M}
  \left\{
   t^{(a)}_{1}= 1
   \right\} . 
\end{equation*}

Let $\mathcal{S}$ be the rank one local system  determined by $U(t)$ and 
 $\mathcal{S}^*$, the dual local system of $\mathcal{S}$. 
The hypergeometric paring between the twisted homology group and  twisted de Rham cohomology group 
 is 
\begin{align*}
H_{(L-1)M}(T, \mathcal{S}^*)\times H^{(L-1)M}(T,\nabla)&\longrightarrow \C
\\
\left(
\Delta, \varphi
\right)&\longmapsto \int_\Delta U(t)\varphi, 
\end{align*}
where $\varphi$ is a rational $(L-1)M$-form holomorphic outside $D$ and $\nabla$ 
is the covariant differential operator given by 
$\nabla=d+d\log(U(t))\wedge$.

  Denote by $\frak{S}_M^{L-1}$, ($L-1$)-th products of the symmetric group with the degree $M$. 
Let the action of $\frak{S}_M^{L-1}$ on a rational function $f(t)$ of variables $t=(t_1^{(1)},\ldots,t_{L-1}^{(1)},\ldots, 
t_1^{(M)},\ldots,t_{L-1}^{(M)})$ be defined by 
\begin{equation}\label{eq sigma}
\sigma(f(t))=f(t_1^{(\sigma_1(1))},\ldots,t_{L-1}^{(\sigma_{L-1}(1))},
\ldots, t_1^{(\sigma_1(M))},\ldots,t_{L-1}^{(\sigma_{L-1}(M))})
\end{equation}
for $\sigma=(\sigma_1,\ldots,\sigma_{L-1})\in\frak{S}_M^{L-1}$. Let $\mathrm{Sym}[f(t)]$ be the symmetrization 
of $f(t)$, given  by $\mathrm{Sym}[f(t)]=\sum_{\sigma\in\frak{S}_M^{L-1}}\sigma(f(t))$.

\begin{dfn}\label{def integral formula}
For $M\in\mathbb{Z}_{\ge 2}$, 
we define an integral formula by 
\begin{equation*}
\Psi_M\left(
{\bf q,z}
\right)=
\int_\Delta U(t)\cdot\mathrm{Sym}\left[
\prod_{a=1}^M
 \left(f_0(t^{(a)})-\sum_{i=1}^N\sum_{n=1}^{L-1}
 f_n^{(i)}(t^{(a)})q_n^{(i)}\right)\right]dt, 
\end{equation*}
where $\Delta\in H_{(L-1)M}(T, \mathcal{S}^*)$ and 
\begin{align*}
&f_0(t^{(a)})=\prod_{m=1}^{L-1}\frac{1}{t^{(a)}_{m-1}-t^{(a)}_m}, \quad f^{(i)}_n(t^{(a)})=\frac{1}{1-z_it^{(a)}_{L-1}}\prod_{m=1,\atop m\neq n}^{L-1}\frac{1}{t^{(a)}_{m-1}-t^{(a)}_m},
\quad t_0^{(a)}=1,
\\
&dt=dt^{(1)}_1\wedge\cdots \wedge dt^{(1)}_{L-1}\wedge dt^{(2)}_{1}\wedge\cdots \wedge dt^{(2)}_{L-1}\wedge
\cdots\wedge dt^{(m)}_{1}\wedge \cdots \wedge
dt^{(m)}_{L-1}.
\end{align*}

\end{dfn}

\begin{thm}\label{thm IF}
If $\ka_0-\sum_{i=1}^N\theta_i=M$ and $\ka_n=1$ ($2\le n\le L-1$), then  
the integral formula $\Psi_M({\bf q,z})$ is a solution to the Schr\"odinger system \eqref{eq Schrodinger},  with  
\begin{equation*}
\alpha_n=e_{n+1}-e_n+1,\quad
\beta_i=-\theta_i, \quad \gamma=\ka_1+M-1,
\end{equation*}
for $1\le n\le L-1$ and $1\le i\le N$, where $e_L=e_0$. 
\end{thm}

For $A\in\mathcal{A}_M$, 
let $\varphi_A(t)$ be the rational $(L-1)M$-form  holomorphic outside $D$ 
defined by 
\begin{equation*}
 \varphi_{A}(t)
=\mathrm{Sym}\left[
(-1)^{M-A_0}\begin{pmatrix}
M\\
A
\end{pmatrix}
\prod_{i=1}^N\prod_{n=1}^{L-1}
 \prod_{a=S_{n-1}^{(i)}+1}^{S_n^{(i)}}
f^{(i)}_n\left(
t^{(a)}
\right)\prod_{a=M-A_0+1}^Mf_0\left(
t^{(a)}
\right)\right]dt,
\end{equation*}
where 
\begin{align*}
&A_0=M-\sum_{1\le i\le N, \atop 1\le n\le L-1}A_{n,i}, \quad \begin{pmatrix}
M\\
A
\end{pmatrix}
=\frac{M!}{A_0!\prod_{1\le i\le N, \atop 1\le n\le L-1}A_{n,i}!}, 
\quad
S_n^{(i)}=\sum_{j=1}^{i-1}\sum_{m=1}^{L-1}A_{m,j}+\sum_{m=1}^nA_{m,i}.
\end{align*}
Then, 
 the integral formula  is expressed as  
\begin{equation*}
\Psi_M\left(
{\bf q,z}
\right)=\sum_{A\in\mathcal{A}_M}q^A\int_\Delta U(t) \varphi_A(t). 
\end{equation*}

Since, in general,  it holds for an $(L-1)M$-form $\varphi$ that
\begin{equation*}
\frac{\partial}{\partial z_i}\int_\Delta U \varphi=\int_\Delta U \left(
\frac{1}{U}\frac{\partial U}{\partial z_i}\varphi+\frac{\partial \varphi}{\partial z_i}
\right), 
\end{equation*}
let a linear operator $\nabla_i$ ($i=1,\ldots,N$) acting on 
$\varphi$ be defined as
\begin{equation*}
\nabla_i\varphi=\frac{1}{U}\frac{\partial  U}{\partial z_i}\varphi+\frac{\partial \varphi}{\partial z_i}.
\end{equation*}

Let us explain our proof of Theorem \ref{thm IF} briefly. 
 We compute $\ka\nabla_i\varphi_A(t)$ and obtain the  
 linear Pfaffian system for $\{\int_\Delta U \varphi_A(t)|A\in\mathcal{A}_M\}$. 
While we compute the action of the Hamiltonians $H_i$ on $q^A$ 
and obtain the coefficient of $q^A$ of $H_i\Psi_M({\bf q,z})$ 
as a linear combination of elements of $\{\int_\Delta U \varphi_A(t)|A\in\mathcal{A}_M\}$. 
Finally, comparing 
 both results, we obtain Theorem \ref{thm IF}.

{\bf A proof of Theorem \ref{thm IF}}

Fix $i\in\{1,\ldots, N\}$ and $A\in\mathcal{A}_M$. We compute $\nabla_i \varphi_{A}(t)$ as follows. First, 
we have 
\begin{equation}\label{first eq}
\ka\nabla_i \varphi_{A}(t)=\mathrm{Sym}\left[
\left(
\beta_i\sum_{j=1,\atop j\neq i}^N\sum_{n=1}^{L-1}\frac{A_{n,j}t_{L-1}^{\left(
S_n^{(j)}
\right)}}{1-z_it_{L-1}^{\left(
S_n^{(j)}
\right)}}
+\beta_i\frac{A_0t_{L-1}^{\left(
M-A_0+1
\right)}}{1-z_it_{L-1}^{\left(M-A_0+1
\right)}}
+ \left(\beta_i+ \ka\right)\sum_{n=1}^{L-1}\frac{A_{n,i}t_{L-1}^{\left(
S_n^{(i)}
\right)}}{1-z_it_{L-1}^{\left(S_n^{(i)}\right)}}
\right)\bar{\varphi}_A(t)\right]dt, 
\end{equation}
where $\bar{\varphi}_A(t)$ is defined by 
\begin{equation*}
\bar{\varphi}_A(t)=(-1)^{M-A_0}\begin{pmatrix}
M\\
A
\end{pmatrix}
\prod_{i=1}^N\prod_{n=1}^{L-1}
 \prod_{a=S_{n-1}^{(i)}+1}^{S_n^{(i)}}
f^{(i)}_n\left(
t^{(a)}
\right)\prod_{a=M-A_0+1}^Mf_0\left(
t^{(a)}
\right).  
\end{equation*}

 Using a relation
\begin{equation}\label{Jacobi identity}
\frac{t}{1-z_it}=\frac{1-z_jt}{z_i-z_j}\left(
\frac{1}{1-z_it}-\frac{1}{1-z_jt}
\right), 
\end{equation}
we get
\begin{equation*}
\text{the first term of \eqref{first eq}}
=\beta_i\sum_{j=1,\atop j\neq i}^N\sum_{n=1}^{L-1}\frac{1}{z_i-z_j}
\left((A_{n,i}+1) \varphi_{\left(A_{n,j}-1, A_{n,i}+1\right)}(t)
-A_{n,j}
 \varphi_{A}(t)\right), 
\end{equation*}
where $\varphi_{\left(A_{n,j}-1, A_{n,i}+1\right)}(t)$ is the rational $(L-1)M$-form defined for 
the matrix in $\mathcal{A}_M$ whose $(n,j)$ entry is $A_{n,j}-1$ and $(n,i)$ entry is $A_{n,i}+1$, and the other $(m,k)$ entries 
are $A_{m,k}$. 

As for the second term of \eqref{first eq}, 
using a relation 
\begin{equation}\label{eq f0}
\frac{t_{L-1}}{1-z_it_{L-1}}=\frac{1}{(z_i-1)f_0(t)}\left(
-f_0(t)+\sum_{n=1}^{L-1}f^{(i)}_n(t)
\right)
\end{equation}
we obtain 
\begin{equation*}
\text{the second term of \eqref{first eq} }
=\frac{-\beta_i}{z_i-1}\left(A_0
 \varphi_{A}(t)
+\sum_{n=1}^{L-1}
\left(
A_{n,i}+1
\right)
 \varphi_{\left( A_{n,i}+1\right)}(t)
\right),
\end{equation*}
where $\varphi_{\left(A_{n,i}+1\right)}(t)$ is the rational $(L-1)M$-form defined for 
the matrix in $\mathcal{A}_M$ whose  $(n,i)$ entry is $A_{n,i}+1$, and the other $(m,k)$ entries 
are $A_{m,k}$. 

In order to calculate the third term of \eqref{first eq}, 
we compute coboundaries $X_n^{(i)}$ ($n=1,\ldots, L-1$) defined by 
\begin{equation}\label{eq coboundary}
X_n^{(i)}=\ka\sum_{m=n}^{L-1}\nabla\left(\sum_{\sigma\in\frak{S}_M^{L-1}}\sigma\left(
t_{m}^{( (S^{(i)}_n)}
\bar{\varphi}_A(t)\right) \ 
{*dt}_m^{(\sigma_m(S_n^{(i)}))}\right), 
\end{equation}
 where  $*dt_m^{(a)}$ is defined by 
\begin{equation*}
*dt_m^{(a)}=(-1)^{(L-1)(a-1)+m-1}dt_1^{(1)}\wedge  
\cdots \wedge \widehat{dt_m^{(a)}}\wedge \cdots \wedge t_{L-1}^{(M)}, 
\end{equation*}
so that  $dt_m^{(a)}\wedge *dt_m^{(a)}=dt$.

For $m\neq n$, denote by 
 $\varphi_{\left( A_{n,i}-1, A_{m,i}+1\right)}(t)$, the rational $(L-1)M$-form defined for 
the matrix in $\mathcal{A}_M$ whose $(n,i)$ entry is $A_{n,i}-1$ and $(m,i)$ entry is $A_{m,i}+1$, 
and the other $(l,k)$ entries 
are $A_{l,k}$, and denote by $\varphi_{\left(A_{n,i}-1\right)}(t)$,  the rational $(L-1)M$-form defined for 
the matrix in $\mathcal{A}_M$ whose  $(n,i)$ entry is $A_{n,i}-1$, and the other $(l,k)$ entries 
are $A_{l,k}$. 

Using 
the relations \eqref{Jacobi identity} and \eqref{eq f0}, 
 we obtain by straightforward calculations
\begin{align*}
X_n^{(i)}=&\mathrm{Sym}\left[\left(\beta_i+ \ka\right)z_i \frac{t_{L-1}^{(S_n^{(i)})}}{1-z_it_{L-1}^{(S_n^{(i)})}}
\bar{\varphi}_A(t)\right]dt
\\
&+\left(\sum_{m=n}^{L-1}\alpha_m-(L-1-n)\right)\varphi_{A}(t)
+\left(1+\delta_{n,1}(\gamma-1)\right)\sum_{m=n+1}^{L-1}\frac{A_{m,i}+1}{A_{n,i}}
 \varphi_{\left(A_{m,i}+1, A_{n,i}-1 \right)}(t)
\\
&+\left(1+\delta_{n,1}(\gamma-1)\right){1\over z_i-1}\left(\frac{A_0+1}{A_{n,i}}
\varphi_{\left(
A_{n,i}-1
\right)}(t)
+\sum_{m=1,\atop m\neq n}^{L-1}\frac{A_{m,i}+1}{A_{n,i}}
\varphi_{\left(A_{m,i}+1, 
A_{n,i}-1
\right)}(t)+ \varphi_{A}(t)
\right)
\\
&+\sum_{j=1,\atop j\neq i}^N\frac{\beta_jz_j}{z_i-z_j}\left(
 \varphi_{A}(t)-
\frac{A_{n,j}+1}{A_{n,i}}
 \varphi_{\left( 
A_{n,i}-1, A_{n,j}+1 
\right)}(t)
\right)
+Y_{n}^{(i)},
\end{align*}
where 
\begin{equation}\label{eq Y Z}
Y_{n}^{(i)}=\mathrm{Sym}\left[\sum_{m=n}^{L-1}t_{m}^{\left(
S_n^{(i)}
\right)}W_{n,m}^{(i)}\bar{\varphi}_A(t)\right]dt, 
\end{equation}
with 
\begin{equation*}
W^{(i)}_{n,m}=\sum_{a=1,\atop a\neq 
S_n^{(i)}}^M\left(
\frac{-1}{t_m^{\left(
S_n^{(i)}
\right)}
-
t_{m-1}^{(a)}}
+
\frac{2}{t_m^{\left(
S_n^{(i)}
\right)}
-
t_m^{(a)}}
+\frac{-1}{t_m^{\left(
S_n^{(i)}\right)
}
-
t_{m+1}^{(a)}}+\delta_{m,1}\frac{1}{t_{1}^{\left(
S_n^{(i)}
\right)}-1}\right).
\end{equation*}

We compute $Y_n^{(i)}$ in Lemmas \ref{lem l<n Y}, \ref{lem n<l Y}, \ref{lem n=l Y} and \ref{lem l=0 Y}. Owing to  those lemmas, 
we obtain
\begin{align}
&\ka z_i\nabla_i \varphi_A(t)-\sum_{n=1}^{L-1}A_{n,i}X_n^{(i)}=\left\{
-\sum_{n=1}^{L-1}A_{n,i}\left(
\sum_{m=n}^{L-1}\alpha_m+L-n-\beta_i+\sum_{m=1}^nA_{m,i}
\right)\right.\nn
\\
&+\frac{1}{z_i-1}\left(
A_0\left(
\sum_{n=1}^{L-1}A_{n,i}-\beta_i
\right)
-\sum_{j=1}^N\sum_{n=1}^{L-1}A_{n,i}A_{n,j}+A_{1,i}(M-\gamma)
\right)\nn
\\
&\left.+\sum_{j=1,\atop j\neq i}^N\frac{z_j}{z_i-z_j}\sum_{n=1}^{L-1}\left(A_{n,i}\left(
\sum_{m=1}^{L-1}A_{m,j}+A_{n,j}-\beta_j
\right)
-\beta_iA_{n,j}
\right)\right\}\varphi_A(t)\nn
\\
&-\frac{A_0+1}{z_i-1}
\sum_{n=1}^{L-1}\left(\sum_{j=1}^NA_{n,j}+\delta_{n,1}(\gamma-M)\right)\varphi_{(A_{n,i}-1)}(t)
+\frac{z_i}{z_i-1}\left(\sum_{m=1}^{L-1}A_{m,i}-\beta_i\right)\sum_{n=1}^{L-1}(A_{n,i}+1)\varphi_{(A_{n,i}+1)}(t)\nn
\\
&-\frac{1}{z_i-1}\sum_{n=1}^{L-1}\left(\sum_{j=1}^NA_{n,j}+\delta_{n,1}(\gamma-M)\right)
\left(\sum_{m=1}^{n-1}
(A_{m,i}+1)
  \varphi_{\left(A_{m,i}+1, A_{n,i}-1 \right)}(t) 
  +z_i\sum_{m=n+1}^{L-1}
(A_{m,i}+1)
  \varphi_{\left(A_{m,i}+1, A_{n,i}-1 \right)}(t)\right)\nn
  \\
  &+\sum_{j=1,\atop j\neq i}^{N}\frac{A_{n,j}+1}{z_i-z_j}
\left(z_j\sum_{m=1}^{n-1}
(A_{m,i}+1)
  \varphi_{\left(A_{m,j}-1, A_{m,i}+1, A_{n,i}-1, A_{n,j}+1 \right)}(t) 
  +z_i\sum_{m=n+1}^{L-1}
(A_{m,i}+1)
  \varphi_{\left(A_{m,j}-1, A_{m,i}+1, A_{n,i}-1,A_{n,j}+1 \right)}(t)
   \right) \nn
  \\
  &+\sum_{j=1,\atop j\neq i}^{N}\frac{z_i}{z_i-z_j}\left(\beta_i-\sum_{m=1}^{L-1}A_{m,i}\right)
  \sum_{n=1}^{L-1}(A_{n,i}+1)\varphi_{(A_{n,j}-1, A_{n,i}+1)}(t)\nn
  \\
  &+\sum_{j=1,\atop j\neq i}^{N}\frac{z_j}{z_i-z_j}\left(\beta_j-\sum_{m=1}^{L-1}A_{m,j}\right)
  \sum_{n=1}^{L-1}(A_{n,j}+1)\varphi_{(A_{n,i}-1, A_{n,j}+1)}(t), 
\label{eq thm final result}
\end{align}
where the rational $(L-1)M$-form $\varphi_{\left(A_{m,j}-1, A_{m,i}+1, A_{n,i}-1, A_{n,j}+1 \right)}(t)$ be  defined for 
the matrix in $\mathcal{A}_M$ whose $(m,j)$, $(m,i)$, $(n,i)$, $(n,j)$ entries are $A_{m,j}-1$,   $A_{m,i}+1$,  $A_{n,i}-1$, 
and  $A_{n,j}+1$, respectively, and the other $(l,k)$ entries 
are $A_{l,k}$.
Hence, for $A\in\mathcal{A}_M$, as an element in the twisted de Rham cohomology group $H^{(L-1)M}(T, \nabla)$, 
$\ka\nabla_i\varphi_A(t)$ is expressed in terms of elements of $\{\varphi_B(t)|B\in \mathcal{A}_M\}$. 

On the other hand, computations of the action of the Hamiltonian $H_i$ on $q^A$ for $A\in\mathcal{A}_M$ 
is straightforward and it is easy to see that the coefficient of $q^A$ of $H_i \Psi_M({\bf q,z})$ is equal to 
the hypergeometric pairing between the cycle $\Delta\in H_{(L-1)M}(T,\mathcal{S}^*)$ and the right hand 
side of \eqref{eq thm final result}. Therefore, we complete our proof. 
\qed

\subsection{Lemmas}

Through lemmas below,  fix $1\le n\le L-1$, $1\le i\le  N$ and $A\in\mathcal{A}_M$. 
For a triple $(n,i,A)$, the coboundary $X_n^{(i)}$ is defined by \eqref{eq coboundary} 
and  expressed as 
a linear combination of elements in $\{
\varphi_B(t) |B\in\mathcal{A}_M\}$, and $Y_n^{(i)}$. In this subsection, 
we compute $Y_n^{(i)}$, so that we show that they are also expressed as 
a linear combination of elements in 
$\{\varphi_B(t) |B\in\mathcal{A}_M\}$. 

We divide $Y_n^{(i)}$  as 
\begin{equation*}
Y_n^{(i)}=\sum_{1\le j\le N, \atop 1\le l \le L-1 }\left(Y_n^{(i)}\right)_{l,j}+\left(Y_n^{(i)}\right)_0 
\end{equation*}
and we compute  $\left(Y_n^{(i)}\right)_{l,j}$ and $\left(Y_n^{(i)}\right)_0$, 
where for  $l\neq n$ or $j\neq i$, 
\begin{align*}
&\left(Y_n^{(i)}\right)_{l,j}=\mathrm{Sym}\left[A_{l,j}C\left(n, S_n^{(i)}, S_l^{(j)}\right)
\bar{\varphi}_A(t)\right]dt, 
\\
&\left(Y_n^{(i)}\right)_{n,i}=\mathrm{Sym}\left[(A_{n,i}-1)C\left(n, S_n^{(i)}, S_n^{(i)}-1\right)\bar{\varphi}_A(t)\right]dt, 
\\
&\left(Y_n^{(i)}\right)_0=\mathrm{Sym}\left[A_0C\left(n, S_n^{(i)}, M-A_0+1\right)\bar{\varphi}_A(t)\right]dt. 
\end{align*}
Here,  for $1\le a\neq b\le M$,   
\begin{equation*}
C(n,a,b)=\sum_{m=n}^{L-1}t_m^{(a)}\left(
\frac{-1}{t_{m}^{(a)}-t_{m-1}^{(b)}}+\frac{2}{t_{m}^{(a)}-t_{m}^{(b)}}
+\frac{-1}{t_{m}^{(a)}-t_{m+1}^{(b)}}\right)+\delta_{n,1}\frac{t_{1}^{(a)}}{t_{1}^{(a)}-1}. 
\end{equation*}

 Let the rational functions $f_{l,m}^{(j)}(t^{(a)})$ be defined by 
\begin{equation*}
f_{l,m}^{(j)}(t^{(a)})={1\over 1-z_jt_{L-1}^{(a)}}\sum_{k=1,\atop k\neq l,m}^{L-1}\frac{1}{t_{k-1}^{(a)}-t_k^{(a)}}. 
\end{equation*}

\begin{lem}\label{lem l<n Y}
When $1\le l< n$, for $1\le j\neq i\le N$, we have 
\begin{align*}
\left(Y_n^{(i)}\right)_{l,j}=&\left(A_{l,i}+1\right)  \varphi_{\left(A_{l,j}-1, A_{l,i}+1 \right)}(t) 
+\frac{z_j}{z_i-z_j}\left(\left(
A_{l,i}+1
\right)
  \varphi_{\left(A_{l,j}-1, A_{l,i}+1 \right)}(t) 
\right.
\\
&-A_{l,j}  \varphi_{A }(t) 
-\frac{\left(
A_{l,i}+1
\right)
\left(
A_{n,j}+1
\right)}{A_{n,i}}
  \varphi_{\left(A_{l,j}-1, A_{l,i}+1, A_{n,i}-1, A_{n,j}+1 \right)}(t) 
\left.+\frac{\left(
A_{n,j}+1
\right)A_{l,j}}{A_{n,i}}
  \varphi_{\left(A_{n,j}+1, A_{n,i}-1 \right)}(t) 
\right),
\end{align*}
and 
for $j=i$, we have
\begin{equation*}
\left(Y_n^{(i)}\right)_{l,i}=A_{l,i}  \varphi_{A}(t) . 
\end{equation*}
\end{lem}
\begin{proof}
It suffices to show that 
\begin{align}
&\mathrm{Sym}\left[
C(n,1,2){1\over t_{L-1}^{(1)}}f_n^{(i)}(t^{(1)}){1\over t_{L-1}^{(2)}}f_l^{(j)}(t^{(2)})
\right]\nn
\\
&=\mathrm{Sym}\left[
{1\over t_{L-1}^{(1)}}f_n^{(i)}(t^{(1)}){1\over t_{L-1}^{(2)}}f_l^{(i)}(t^{(2)})
\right]\nn
\\
&+{z_j\over z_i-z_j}\mathrm{Sym}\left[
{1\over t_{L-1}^{(1)}}{1\over t_{L-1}^{(2)}}\left(f_n^{(i)}(t^{(1)})-f_n^{(j)}(t^{(1)})\right)
\left(
f_l^{(i)}(t^{(2)})-f_l^{(j)}(t^{(2)})\right)
\right], \label{eq l<n Y}
\end{align}
where the symmetrization $\mathrm{Sym}[f(t)]$ stands for $\sum_{\sigma\in\frak{S}_2^{L-1}}\sigma(f(t))$ 
(see \eqref{eq sigma}), the rational functions $f_n^{(i)}(t^{(a)})$ are defined in Definition \ref{def integral formula}, and 
if $j=i$, then 
we understand that the second line of the right hand side of \eqref{eq l<n Y} is vanished.

Firstly, we claim that for $n\le k \le L-2$, we have 
\begin{align}
&\mathrm{Sym}\left[
\sum_{m=n}^k t_m^{(1)}\left(
\frac{-1}{t_{m}^{(1)}-t_{m-1}^{(2)}}+\frac{2}{t_{m}^{(1)}-t_{m}^{(2)}}
+\frac{-1}{t_{m}^{(1)}-t_{m+1}^{(2)}}\right){1\over t_{L-1}^{(1)}}f_n^{(i)}(t^{(1)}){1\over t_{L-1}^{(2)}}f_l^{(j)}(t^{(2)})
\right]\nn
\\
&=\mathrm{Sym}\left[ 
{1\over \left(t_k^{(1)}-t_{k+1}^{(2)}\right)}{ 1\over \left(  t_{k+1}^{(1)}-t_{k}^{(2)} \right)}{t_{k+1}^{(1)}\over t_{L-1}^{(1)}}
f_n^{(i)}(t^{(1)}){1\over t_{L-1}^{(2)}}f_{l,k+1}^{(j)}(t^{(2)})
\right]. \label{eq l<n claim}
\end{align}

We show \eqref{eq l<n claim} by induction. Let $k=n$. then,  we have 
\begin{align}
&\mathrm{Sym}\left[\left(\frac{-1}{t_{n}^{(1)}-t_{n-1}^{(2)}}+\frac{1}{t_{n}^{(1)}-t_{n}^{(2)}}
\right){t_{n}^{(1)}\over t_{L-1}^{(1)}}f_n^{(i)}(t^{(1)}){1\over t_{L-1}^{(2)}}f_l^{(j)}(t^{(2)})
\right]\nn
\\
&=\mathrm{Sym}\left[
\frac{1}{t_{n-1}^{(2)}-t_{n}^{(1)}}\frac{1}{t_{n}^{(1)}-t_{n}^{(2)}}{t_{n}^{(1)}\over t_{L-1}^{(1)}}f_n^{(i)}(t^{(1)}){1\over t_{L-1}^{(2)}}f_{l,n}^{(j)}(t^{(2)})
\right]\nn
\\
&=\mathrm{Sym}\left[
\frac{-1}{t_{n}^{(1)}-t_{n}^{(2)}}\frac{1}{t_{n}^{(2)}-t_{n+1}^{(1)}}\frac{1}{t_{n}^{(1)}-t_{n+1}^{(2)}}{t_{n}^{(2)}\over t_{L-1}^{(1)}}
f_{n,n+1}^{(i)}(t^{(1)}){1\over t_{L-1}^{(2)}}f_{l,n+1}^{(j)}(t^{(2)})
\right],\label{eq l<n n1}
\end{align}
where in the last line, we interchange $t_n^{(1)}$ with $t_n^{(2)}$ and 
\begin{align}
&\mathrm{Sym}\left[\left(\frac{1}{t_{n}^{(1)}-t_{n}^{(2)}}+\frac{-1}{t_{n}^{(1)}-t_{n+1}^{(2)}}
\right){t_{n}^{(1)}\over t_{L-1}^{(1)}}f_n^{(i)}(t^{(1)}){1\over t_{L-1}^{(2)}}f_l^{(j)}(t^{(2)})
\right]\nn
\\
&=\mathrm{Sym}\left[\frac{1}{t_{n}^{(1)}-t_{n}^{(2)}}\frac{1}{t_{n}^{(1)}-t_{n+1}^{(2)}}
{t_{n}^{(1)}\over t_{L-1}^{(1)}}f_n^{(i)}(t^{(1)}){1\over t_{L-1}^{(2)}}f_{l,n+1}^{(j)}(t^{(2)})
\right]. \label{eq l<n n2}
\end{align}
Thus, the left hand side of \eqref{eq l<n claim} for $k=n$, that is, \eqref{eq l<n n1} plus \eqref{eq l<n n2},  
becomes the right hand side of \eqref{eq l<n claim} for $k=n$. 

Suppose \eqref{eq l<n claim} holds for $k-1$, then 
\begin{align*}
&\mathrm{Sym}\left[\left(t_{k}^{(1)}\left(\frac{-1}{t_{k}^{(1)}-t_{k-1}^{(2)}}+\frac{1}{t_{k}^{(1)}-t_{k}^{(2)}}\right)+
\sum_{m=n}^{k-1} t_{m}^{(1)}\left(\frac{-1}{t_{m}^{(1)}-t_{m-1}^{(2)}}+\frac{2}{t_{m}^{(1)}-t_{m}^{(2)}}
+\frac{-1}{t_{m}^{(1)}-t_{m+1}^{(2)}}\right)\right){1\over t_{L-1}^{(1)}}f_n^{(i)}(t^{(1)}){1\over t_{L-1}^{(2)}}f_l^{(j)}(t^{(2)})
\right]
\nn
\\
&=\mathrm{Sym}\left[
\frac{1}{t_{k}^{(1)}-t_{k}^{(2)}}\frac{1}{t_{k-1}^{(2)}-t_{k}^{(1)}}\frac{1}{t_{k-1}^{(1)}-t_{k}^{(2)}}{t_{k}^{(1)}\over t_{L-1}^{(1)}}
f_{n,k}^{(i)}(t^{(1)}){1\over t_{L-1}^{(2)}}f_{l,k}^{(j)}(t^{(2)})
\right]
\\
&=\mathrm{Sym}\left[
\frac{-1}{t_{k}^{(1)}-t_{k}^{(2)}}\frac{1}{t_{k}^{(2)}-t_{k+1}^{(1)}}\frac{1}{t_{k}^{(1)}-t_{k+1}^{(2)}}{t_{k}^{(2)}\over t_{L-1}^{(1)}}
f_{n,k+1}^{(i)}(t^{(1)}){1\over t_{L-1}^{(2)}}f_{l,k+1}^{(j)}(t^{(2)})
\right],
\end{align*}
where in the last line, we interchange $t_k^{(1)}$ with $t_k^{(2)}$. Since 
\begin{align*}
&\mathrm{Sym}\left[\left(\frac{1}{t_{k}^{(1)}-t_{k}^{(2)}}+\frac{-1}{t_{k}^{(1)}-t_{k+1}^{(2)}}
\right){t_{k}^{(1)}\over t_{L-1}^{(1)}}f_n^{(i)}(t^{(1)}){1\over t_{L-1}^{(2)}}f_l^{(j)}(t^{(2)})
\right]\nn
\\
&=\mathrm{Sym}\left[\frac{1}{t_{k}^{(1)}-t_{k}^{(2)}}\frac{1}{t_{k}^{(1)}-t_{k+1}^{(2)}}{t_{k}^{(1)}\over t_{L-1}^{(1)}}
f_n^{(i)}(t^{(1)}){1\over t_{L-1}^{(2)}}f_{l,k+1}^{(j)}(t^{(2)})
\right], 
\end{align*}
the left hand side of \eqref{eq l<n claim} for $k$ becomes the right hand side of \eqref{eq l<n claim} for $k$.

Secondly,  using \eqref{eq l<n claim} for $k=L-2$, we have 
\begin{align*}
&\mathrm{Sym}\left[\left(t_{L-1}^{(1)}\left(\frac{-1}{t_{L-1}^{(1)}-t_{L-2}^{(2)}}+\frac{1}{t_{L-1}^{(1)}-t_{L-1}^{(2)}}\right)+
\sum_{m=n}^{L-2}t_{m}^{(1)} \left(\frac{-1}{t_{m}^{(1)}-t_{m-1}^{(2)}}+\frac{2}{t_{m}^{(1)}-t_{m}^{(2)}}
+\frac{-1}{t_{m}^{(1)}-t_{m+1}^{(2)}}\right)\right){1\over t_{L-1}^{(1)}}f_n^{(i)}(t^{(1)}){1\over t_{L-1}^{(2)}}f_l^{(j)}(t^{(2)})
\right]
\nn
\\
&=\mathrm{Sym}\left[
\frac{1}{t_{L-1}^{(1)}-t_{L-1}^{(2)}}\frac{1}{t_{L-2}^{(2)}-t_{L-1}^{(1)}}\frac{1}{t_{L-2}^{(1)}-t_{L-1}^{(2)}}
f_{n,L-1}^{(i)}(t^{(1)}){1\over t_{L-1}^{(2)}}f_{l,L-1}^{(j)}(t^{(2)})
\right]
\\
&=\mathrm{Sym}\left[
\frac{-1}{t_{L-1}^{(1)}-t_{L-1}^{(2)}}
f_{n}^{(j)}(t^{(1)}){1\over t_{L-1}^{(1)}}f_{l}^{(i)}(t^{(2)})
\right],
\end{align*}
where in the last line, we interchange $t_{L-1}^{(1)}$ with $t_{L-1}^{(2)}$. Hence, the left hand side of \eqref{eq l<n Y} 
is equal to 
\begin{align*}
&\mathrm{Sym}\left[
\frac{-1}{t_{L-1}^{(1)}-t_{L-1}^{(2)}}
f_{n}^{(j)}(t^{(1)}){1\over t_{L-1}^{(1)}}f_{l}^{(i)}(t^{(2)})+\frac{1}{t_{L-1}^{(1)}-t_{L-1}^{(2)}}
f_n^{(i)}(t^{(1)}){1\over t_{L-1}^{(2)}}f_{l}^{(j)}(t^{(2)})
\right]\nn
\\
&=\mathrm{Sym}\left[{1\over t_{L-1}^{(1)}}{1\over t_{L-1}^{(2)}}
f_{n}^{(i)}(t^{(1)})f_{l}^{(i)}(t^{(2)})\frac{1-z_j(t_{L-1}^{(1)}+t_{L-1}^{(2)})+z_iz_jt_{L-1}^{(1)}t_{L-1}^{(2)}}{(1-z_jt_{L-1}^{(1)})(1-z_jt_{L-1}^{(2)})}
\right]. 
\end{align*}
Therefore, the relation \eqref{eq l<n Y} holds.

\end{proof}

\begin{lem}\label{lem n<l Y}
When $n < l\le L-1$,   for $1\le j\neq i\le N$, we have 
\begin{align*}
\left(Y_n^{(i)}\right)_{l,j}=&-A_{l,j}\left( \frac{\delta_{n,1}}{z_i-1}+ \frac{z_j}{z_i-z_j} \right)  \varphi_{A}(t)
+
\left(
A_{l,i}+1
\right)\frac{z_i}{z_i-z_j}
  \varphi_{\left(A_{l,j}-1, A_{l,i}+1 \right)}(t) 
\\
& 
-\frac{\left(
A_{l,i}+1
\right)
\left(
A_{n,j}+1
\right)}{A_{n,i}}\frac{z_i}{z_i-z_j}
  \varphi_{\left(A_{l,j}-1, A_{l,i}+1, A_{n,i}-1, A_{n,j}+1\right)}(t) 
+
\frac{\left(
A_{n,j}+1
\right)A_{l,j}}{A_{n,i}} \frac{z_j}{z_i-z_j} \varphi_{\left(A_{n,j}+1, A_{n,i}-1 \right)}(t) 
\\
&-\frac{A_{l,j}}{A_{n,i}}\frac{\delta_{n,1}}{(z_i-1)}\left(\left(
A_0+1
\right)
  \varphi_{\left( A_{n,i}-1 \right)}(t) 
+z_i\sum_{m=2}^{L-1}\left(
A_{m,i}+1
\right) \varphi_{\left(A_{m,i}+1, A_{1,i}-1 \right)}(t) 
\right),
\end{align*}
and 
for $j=i$, we have 
\begin{align*}
\left(Y_n^{(i)}\right)_{l,i}=&
-\frac{\delta_{n,1}A_{l,i}}{(z_i-1)}\left(  \varphi_{A}(t)+{\left(
A_0+1
\right)\over A_{1,i}}
  \varphi_{\left( A_{1,i}-1 \right)}(t) 
+z_i\sum_{m=2}^{L-1}{\left(
A_{m,i}+1
\right) \over A_{1,i} }\varphi_{\left(A_{m,i}+1, A_{1,i}-1 \right)}(t) 
\right) 
\end{align*}
\end{lem}
\begin{proof}
It suffices to show that for $n\ge 2$, 
\begin{align}
&\mathrm{Sym}\left[
C(n,1,2){1\over t_{L-1}^{(1)}}f_n^{(i)}(t^{(1)}){1\over t_{L-1}^{(2)}}f_l^{(j)}(t^{(2)})
\right]\nn
\\
&=\frac{1}{z_i-z_j}\mathrm{Sym}\left[{1\over t_{L-1}^{(1)}}{1\over t_{L-1}^{(2)}}
\left(f_n^{(i)}(t^{(1)})-f_n^{(j)}(t^{(1)})\right)\left(z_if_l^{(i)}(t^{(2)})-z_jf_l^{(j)}(t^{(2)})\right)
\right], \label{eq n<l Y}
\end{align}
and for $n=1$, 
\begin{align}
&\mathrm{Sym}\left[
C(1,1,2){1\over t_{L-1}^{(1)}}f_1^{(i)}(t^{(1)}){1\over t_{L-1}^{(2)}}f_l^{(j)}(t^{(2)})
\right]\nn
\\
&=\frac{1}{z_i-z_j}\mathrm{Sym}\left[{1\over t_{L-1}^{(1)}}{1\over t_{L-1}^{(2)}}
\left(f_1^{(i)}(t^{(1)})-f_1^{(j)}(t^{(1)})\right)\left(z_if_l^{(i)}(t^{(2)})-z_jf_l^{(j)}(t^{(2)})\right)
\right],\nn
\\
&+\frac{1}{z_i-1}\mathrm{Sym}\left[{1\over t_{L-1}^{(1)}}\left(  
f_0(t^{(1)})-f_1^{(i)}(t^{(1)})-z_i\sum_{m=2}^{L-1}f_m^{(i)}(t^{(1)})
 \right)
{1\over t_{L-1}^{(2)}}f_l^{(j)}(t^{(2)})
\right], \label{eq 1<l  Y}
\end{align}
where the symmetrization $\mathrm{Sym}[f(t)]$ stands for $\sum_{\sigma\in\frak{S}_2^{L-1}}\sigma(f(t))$ 
(see \eqref{eq sigma}),  and 
if $j=i$, then 
we understand that the right hand side of \eqref{eq n<l Y} and the first line of the right hand side of 
\eqref{eq 1<l  Y} is zero. 

We shall show \eqref{eq n<l Y}. Firstly, using \eqref{eq l<n claim} for $k=l-2$, we have 
\begin{align}
&\mathrm{Sym}\left[\left(
\sum_{m=n}^{l-2}t_m^{(1)}\left(
\frac{-1}{t_{m}^{(1)}-t_{m-1}^{(2)}}+\frac{2}{t_{m}^{(1)}-t_{m}^{(2)}}
+\frac{-1}{t_{m}^{(1)}-t_{m+1}^{(2)}}\right)+t_{l-1}^{(1)}\left(\frac{-1}{t_{l-1}^{(1)}-t_{l-2}^{(1)}}+\frac{1}{t_{l-1}^{(1)}-t_{l-1}^{(2)}}\right)
\right)
{1\over t_{L-1}^{(1)}}f_n^{(i)}(t^{(1)}){1\over t_{L-1}^{(2)}}f_l^{(j)}(t^{(2)})
\right]\nn
\\
&=\mathrm{Sym}\left[
\frac{1}{t_{l-1}^{(1)}-t_{l-1}^{(2)}}\frac{1}{t_{l-2}^{(1)}-t_{l-1}^{(2)}}\frac{1}{t_{l-2}^{(2)}-t_{l-1}^{(1)}}
{t_{l-1}^{(1)}\over t_{L-1}^{(1)}}f_{n,l-1}^{(i)}(t^{(1)}){1\over t_{L-1}^{(2)}}f_{l-1,l}^{(j)}(t^{(2)})
\right]\nn
\\
&=\mathrm{Sym}\left[
\frac{-1}{t_{l-1}^{(1)}-t_{l-1}^{(2)}}\frac{1}{t_{l-1}^{(2)}-t_{l}^{(1)}}
{t_{l-1}^{(2)}\over t_{L-1}^{(1)}}f_{n,l}^{(i)}(t^{(1)}){1\over t_{L-1}^{(2)}}f_{l}^{(j)}(t^{(2)})
\right],\label{eq n<l Y 1}
\end{align}
where in the last line, we interchange $t_{l-1}^{(1)}$ with $t_{l-1}^{(2)}$. Thus, \eqref{eq n<l Y 1} is equal to 
\begin{equation*}
-\mathrm{Sym}\left[
\left(  \frac{t_{l-1}^{(1)}}{t_{l-1}^{(1)}-t_{l-1}^{(2)}}+\frac{-t_{l}^{(1)}}{t_{l}^{(1)}-t_{l-1}^{(2)}} \right)
{1\over t_{L-1}^{(1)}}f_n^{(i)}(t^{(1)}){1\over t_{L-1}^{(2)}}f_l^{(j)}(t^{(2)})
\right].
\end{equation*}

Secondly, we claim that for $l\le k\le L-2$, we have 
\begin{align}
&\mathrm{Sym}\left[
\sum_{m=l}^k \left(\frac{-t_{m+1}^{(1)}}{t_{m+1}^{(1)}-t_{m}^{(2)}}+\frac{2t_{m}^{(1)}}{t_{m}^{(1)}-t_{m}^{(2)}}
+\frac{-t_{m-1}^{(1)}}{t_{m-1}^{(1)}-t_{m}^{(2)}}\right){1\over t_{L-1}^{(1)}}f_n^{(i)}(t^{(1)}){1\over t_{L-1}^{(2)}}f_l^{(j)}(t^{(2)})
\right]\nn
\\
&=\mathrm{Sym}\left[ 
{1\over \left(t_k^{(1)}-t_{k+1}^{(2)}\right)}{ 1\over \left(  t_{k}^{(2)}-t_{k+1}^{(1)} \right)}{1\over t_{L-1}^{(1)}}
f_{n,k+1}^{(i)}(t^{(1)}){t_{k+1}^{(2)}\over t_{L-1}^{(2)}}f_{l}^{(j)}(t^{(2)})
\right]. \label{eq n<l claim}
\end{align}
We can prove \eqref{eq n<l claim} by induction and omit the proof of this claim. 

Using \eqref{eq n<l claim} for $k=L-2$, we have
\begin{align*}
&\mathrm{Sym}\left[\left(\frac{t_{L-1}^{(1)}}{t_{L-1}^{(1)}-t_{L-1}^{(2)}}+\frac{-t_{L-2}^{(1)}}{t_{L-2}^{(1)}-t_{L-1}^{(2)}}+
\sum_{m=l}^{L-2} \left(\frac{-t_{m+1}^{(1)}}{t_{m+1}^{(1)}-t_{m}^{(2)}}+\frac{2t_{m}^{(1)}}{t_{m}^{(1)}-t_{m}^{(2)}}
+\frac{-t_{m-1}^{(1)}}{t_{m-1}^{(1)}-t_{m}^{(2)}}\right)\right){1\over t_{L-1}^{(1)}}f_n^{(i)}(t^{(1)}){1\over t_{L-1}^{(2)}}f_l^{(j)}(t^{(2)})
\right]\nn
\\
&=\mathrm{Sym}\left[
\frac{1}{t_{L-1}^{(1)}-t_{L-1}^{(2)}}\frac{1}{t_{L-2}^{(1)}-t_{L-1}^{(2)}}\frac{1}{t_{L-2}^{(2)}-t_{L-1}^{(1)}}
{1\over t_{L-1}^{(1)}}f_{n,L-1}^{(i)}(t^{(1)})f_{l,L-1}^{(j)}(t^{(2)})
\right]
\\
&=\mathrm{Sym}\left[
\frac{-1}{t_{L-1}^{(1)}-t_{L-1}^{(2)}}{1\over t_{L-1}^{(2)}}f_{n}^{(j)}(t^{(1)})f_{l}^{(i)}(t^{(2)})
\right],
\end{align*}
where in the last line, we interchange $t_{L-1}^{(1)}$ with $t_{L-1}^{(2)}$. Hence, the left hand side of \eqref{eq n<l Y} is 
equal to 
\begin{align*}
&\mathrm{Sym}\left[
\frac{-1}{t_{L-1}^{(1)}-t_{L-1}^{(2)}}
f_{n}^{(j)}(t^{(1)}){1\over t_{L-1}^{(2)}}f_{l}^{(i)}(t^{(2)})+\frac{1}{t_{L-1}^{(1)}-t_{L-1}^{(2)}}
f_n^{(i)}(t^{(1)}){1\over t_{L-1}^{(2)}}f_{l}^{(j)}(t^{(2)})
\right]\nn
\\
&=(z_i-z_j)\mathrm{Sym}\left[\frac{1}{1-z_jt_{L-1}^{(1)}}f_n^{(i)}(t^{(1)})\frac{1}{t_{L-1}^{(2)}(1-z_it_{L-1}^{(2)})}
f_l^{(j)}
\right]. 
\end{align*}
Therefore, the relation \eqref{eq n<l Y} holds. 

We shall show \eqref{eq 1<l Y}. We compute the left hand side of \eqref{eq 1<l Y} as follows.  
\begin{align*}
\mathrm{L.H.S.\ of \ \eqref{eq 1<l Y}}=&\mathrm{Sym}\left[\left(\frac{-t_1^{(1)}}{t_1^{(1)}-1}+
C(1,1,2)\right){1\over t_{L-1}^{(1)}}f_1^{(i)}(t^{(1)}){1\over t_{L-1}^{(2)}}f_l^{(j)}(t^{(2)})
\right]
\\
&+\mathrm{Sym}\left[{t_{1}^{(1)}\over t_1^{(1)}-1}{1\over t_{L-1}^{(1)}}f_1^{(i)}(t^{(1)}){1\over t_{L-1}^{(2)}}f_l(t^{(2)})
\right]. 
\end{align*}
The first line of the right hand side of the relation above becomes the first line of the right hand side of \eqref{eq 1<l Y} 
 in the same way of the proof of \eqref{eq n<l Y}. While, we have
\begin{align*}
&\mathrm{Sym}\left[{1\over t_1^{(1)}-1}{t_{1}^{(1)}\over t_{L-1}^{(1)}}f_1^{(i)}(t^{(1)}){1\over t_{L-1}^{(2)}}f_l(t^{(2)})
\right]
\\
&=\mathrm{Sym}\left[{t_{L-1}^{(1)}+\sum_{m=2}^{L-1}(t_{m-1}^{(1)}-t_m^{(1)})
\over t_1^{(1)}-1}{1\over t_{L-1}^{(1)}}f_1^{(i)}(t^{(1)}){1\over t_{L-1}^{(2)}}f_l(t^{(2)})
\right]
\\
&=\mathrm{Sym}\left[{1\over z_i-1}{1\over t_{L-1}^{(1)}}\left(f_0(t^{(1)})-f_1(t^{(1)})-z_i\sum_{m=2}^{L-1}f_m^{(i)}(t^{(1)})\right){1\over t_{L-1}^{(2)}}f_l(t^{(2)})
\right]. 
\end{align*}
Therefore, the relation \eqref{eq 1<l Y} holds. 
\end{proof}

\begin{lem}\label{lem n=l Y}   
 For $1\le j\neq i\le L-1$, we have 
\begin{align*}
\left(Y_n^{(i)}\right)_{n,j}=&
A_{n,j} \left( \frac{1-\delta_{n,1}}{z_i-1}-\frac{2z_j}{z_i-z_j}  \right) \varphi_{A }(t)+
\frac{1-\delta_{n,1}}{z_i-1}\frac{\left(A_0+1\right)A_{n,j}}{A_{n,i}}
  \varphi_{\left(A_{n,i}-1 \right)}(t) 
  \\
  &+\frac{1-\delta_{n,1}}{z_i-1}\frac{A_{n,j}}{A_{n,i}}\left(\sum_{m=1}^{n-1}
(A_{m,i}+1)
  \varphi_{\left(A_{m,i}+1, A_{n,i}-1 \right)}(t) 
  +z_i\sum_{m=n+1}^{L-1}
(A_{m,i}+1)
  \varphi_{\left(A_{m,i}+1, A_{n,i}-1 \right)}(t)\right)
 \\
&+\left(
A_{n,i}+1
\right) \frac{z_i}{z_i-z_j} \varphi_{\left(A_{n,j}-1, A_{n,i}+1 \right)}(t) 
+\frac{\left(
A_{n,j}+1
\right)A_{n,j}}{A_{n,i}}\frac{z_j}{z_i-z_j}\varphi_{\left(A_{n,j}+1, A_{n,i}-1 \right)}(t),  
\end{align*}
and 
for $j=i$, we have
\begin{align*}
\left(Y_n^{(i)}\right)_{n,i}=&
\left(
A_{n,i}-1
\right)\left(\frac{ z_i-\delta_{n,1}}{z_i-1}\right)
  \varphi_{A}(t)+
\frac{ 1-\delta_{n,1}}{z_i-1}\frac{\left(A_0+1\right)\left(A_{n,i}-1\right)}{A_{n,i}}
  \varphi_{\left(A_{n,i}-1 \right)}(t) 
  \\
  &+\frac{1-\delta_{n,1}}{z_i-1}\frac{A_{n,i}-1}{A_{n,i}}\left(\sum_{m=1}^{n-1}
(A_{m,i}+1)
  \varphi_{\left(A_{m,i}+1, A_{n,i}-1 \right)}(t) 
  +z_i\sum_{m=n+1}^{L-1}
(A_{m,i}+1)
  \varphi_{\left(A_{m,i}+1, A_{n,i}-1 \right)}(t)\right). 
\end{align*}
\end{lem}
\begin{proof}
It suffices to show that 
\begin{align}
&\mathrm{Sym}\left[
C(n,1,2){1\over t_{L-1}^{(1)}}f_n^{(i)}(t^{(1)}){1\over t_{L-1}^{(2)}}f_n^{(j)}(t^{(2)})
\right]\nn
\\
&=\frac{1-\delta_{n,1}}{z_i-1}
\mathrm{Sym}\left[{1\over t_{L-1}^{(1)}}\left(-f_0(t^{(1)})+\sum_{m=1}^{n}f_m^{(i)}(t^{(1)})+z_i\sum_{m=n+1}^{L-1}f_m^{(i)}(t^{(1)})\right){1\over t_{L-1}^{(2)}}f_n^{(j)}(t^{(2)})
\right]\nn
\\
&+\mathrm{Sym}\left[
{1\over t_{L-1}^{(1)}}f_n^{(i)}(t^{(1)}){1\over t_{L-1}^{(2)}}f_n^{(i)}(t^{(2)})
\right]\nn
\\
&+{z_j\over z_i-z_j}\mathrm{Sym}\left[
{1\over t_{L-1}^{(1)}}{1\over t_{L-1}^{(2)}}\left(f_n^{(i)}(t^{(1)})-f_n^{(j)}(t^{(1)})\right)
\left(
f_n^{(i)}(t^{(2)})-f_n^{(j)}(t^{(2)})\right)
\right], 
\label{eq l=n Y}
\end{align}
where the symmetrization $\mathrm{Sym}[f(t)]$ stands for $\sum_{\sigma\in\frak{S}_2^{L-1}}\sigma(f(t))$ 
(see \eqref{eq sigma}), and 
if $j=i$, then 
we understand that the third line of the right hand side of \eqref{eq l=n Y} is vanished. 

Firstly,  we have
\begin{align*}
&\mathrm{Sym}\left[\frac{-t_n^{(1)}}{t_n^{(1)}-t_{n-1}^{(2)}}
{1\over t_{L-1}^{(1)}}f_n^{(i)}(t^{(1)}){1\over t_{L-1}^{(2)}}f_n^{(j)}(t^{(2)})
\right]
\\
&=\mathrm{Sym}\left[{t_{L-1}^{(1)}+\sum_{m=n+1}^{L-1}(t_{m-1}^{(1)}-t_m^{(1)})
\over t_n^{(1)}-t_{n-1}^{(1)}}{1\over t_{L-1}^{(1)}}f_n^{(i)}(t^{(1)}){1\over t_{L-1}^{(2)}}f_n^{(j)}(t^{(2)})
\right]
\\
&={1\over z_i-1}\mathrm{Sym}\left[{1\over t_{L-1}^{(1)}}\left(-f_0(t^{(1)})+\sum_{m=1}^{n}f_m^{(i)}(t^{(1)})+z_i\sum_{m=n+1}^{L-1}f_m^{(i)}(t^{(1)})\right){1\over t_{L-1}^{(2)}}f_n^{(j)}(t^{(2)})
\right]. 
\end{align*}

Secondly, we notice that for $n\le m\le L-2$, we have 
\begin{equation*}
\mathrm{Sym}\left[
 \left(\frac{-t_{m+1}^{(1)}}{t_{m+1}^{(1)}-t_{m}^{(2)}}+\frac{2t_{m}^{(1)}}{t_{m}^{(1)}-t_{m}^{(2)}}
+\frac{-t_{m-1}^{(1)}}{t_{m-1}^{(1)}-t_{m}^{(2)}}\right){1\over t_{L-1}^{(1)}}f_n^{(i)}(t^{(1)}){1\over t_{L-1}^{(2)}}f_n^{(j)}(t^{(2)})
\right]=0. 
\end{equation*}

Thirdly, we compute the remaining term as follows. 
\begin{align*}
&\mathrm{Sym}\left[\frac{2}{t_{L-1}^{(1)}-t_{L-1}^{(2)}}
f_n^{(i)}(t^{(1)}){1\over t_{L-1}^{(2)}}f_n^{(j)}(t^{(2)})\right]
\\
&=\mathrm{Sym}\left[
\frac{-1}{t_{L-1}^{(1)}-t_{L-1}^{(2)}}
f_{n}^{(j)}(t^{(1)}){1\over t_{L-1}^{(1)}}f_{n}^{(i)}(t^{(2)})+\frac{1}{t_{L-1}^{(1)}-t_{L-1}^{(2)}}
f_n^{(i)}(t^{(1)}){1\over t_{L-1}^{(2)}}f_{n}^{(j)}(t^{(2)})
\right]\nn
\\
&=\mathrm{Sym}\left[{1\over t_{L-1}^{(1)}}{1\over t_{L-1}^{(2)}}
f_{n}^{(i)}(t^{(1)})f_{n}^{(i)}(t^{(2)})\frac{1-z_j(t_{L-1}^{(1)}+t_{L-1}^{(2)})+z_iz_jt_{L-1}^{(1)}t_{L-1}^{(2)}}{(1-z_jt_{L-1}^{(1)})(1-z_jt_{L-1}^{(2)})}
\right]. 
\end{align*}
Therefore, the relation \eqref{eq l=n Y} holds. 
\end{proof}

\begin{lem}\label{lem l=0 Y}
We have
\begin{align*}
\left(Y_n^{(i)}\right)_0=&
-\frac{1+\delta_{n,1}}{z_i-1}A_0  \varphi_{A}(t)
+\frac{z_i}{z_i-1}
\sum_{m=1}^{L-1}\left(
A_{m,i}+1
\right)  \varphi_{\left(A_{m,i}+1 \right)}(t) 
\\
&-\delta_{n,1}\frac{1}{(z_i-1)}{A_0\left(
A_0+1
\right)\over A_{1,i}}
  \varphi_{\left(A_{1,i}-1 \right)}(t) 
-\delta_{n,1}\frac{z_i}{z_i-1}\sum_{m=2}^{L-1}{A_0\left(
A_{m,i}+1
\right)  \over A_{1,i}}\varphi_{\left(A_{m,i}+1, A_{1,i}-1 \right)}(t). 
\end{align*}
\end{lem}

\begin{proof}
It suffices to show that 
\begin{align}
&\mathrm{Sym}\left[
C(n,1,2){1\over t_{L-1}^{(1)}}f_n^{(i)}(t^{(1)}){1\over t_{L-1}^{(2)}}f_0(t^{(2)})
\right]\nn
\\
&=\mathrm{Sym}\left[{1\over (z_i-1)t_{L-1}^{(1)} t_{L-1}^{(2)}}f_n^{(i)}(t^{(1)})\left(-(1+\delta_{n,1})
f_0(t^{(2)})
+z_i
\sum_{m=1}^{L-1}f_m^{(i)}(t^{(2)})
\right)
\right]\nn
\\
&+\delta_{n,1}\mathrm{Sym}\left[
{1\over(z_i-1) t_{L-1}^{(1)} t_{L-1}^{(2)}}f_0(t^{(2)})\left(f_0(t^{(1)})-z_i\sum_{m=2}^{L-1}f_m^{(i)}(t^{(1)})\right)
\right], \label{eq l=0}
\end{align}
where the symmetrization $\mathrm{Sym}[f(t)]$ stands for $\sum_{\sigma\in\frak{S}_2^{L-1}}\sigma(f(t))$ and 
the rational functions $f_0(t^{(a)})$ are defined in Definition \ref{def integral formula}. 

Firstly, using \eqref{eq l<n claim} for $l=0$ and $k=L-2$, we have 
\begin{align*}
&\mathrm{Sym}\left[\left(C(n,1,2)+\delta_{n,1}{-1\over t_1^{(1)}-1}{t_{1}^{(1)}\over t_{L-1}^{(1)}}\right)
{1\over t_{L-1}^{(1)}}f_n^{(i)}(t^{(1)}){1\over t_{L-1}^{(2)}}f_0(t^{(2)})
\right]
\\
&=\mathrm{Sym}\left[
\frac{-1}{t_{L-1}^{(1)}-t_{L-1}^{(2)}}\frac{1-z_it_{L-1}^{(1)}}{1-z_it_{L-1}^{(2)}}
f_{n}^{(i)}(t^{(1)}){1\over t_{L-1}^{(1)}}f_{0}(t^{(2)})+\frac{1}{t_{L-1}^{(1)}-t_{L-1}^{(2)}}
f_n^{(i)}(t^{(1)}){1\over t_{L-1}^{(2)}}f_{0}(t^{(2)})
\right]
\\
&=
\mathrm{Sym}\left[{1\over t_{L-1}^{(1)}}f_n^{(i)}(t^{(1)})\frac{1}{t_{L-1}^{(2)}(1-z_it_{L-1}^{(2)})}f_{0}(t^{(2)}) 
\right]. 
\end{align*}

Secondly, we have 
\begin{align*}
&\mathrm{Sym}\left[{1\over t_1^{(1)}-1}{t_{1}^{(1)}\over t_{L-1}^{(1)}}f_1^{(i)}(t^{(1)}){1\over t_{L-1}^{(2)}}f_0(t^{(2)})
\right]
\\
&=\mathrm{Sym}\left[{t_{L-1}^{(1)}+\sum_{m=2}^{L-1}(t_{m-1}^{(1)}-t_m^{(1)})
\over t_1^{(1)}-1}{1\over t_{L-1}^{(1)}}f_1^{(i)}(t^{(1)}){1\over t_{L-1}^{(2)}}f_0(t^{(2)})
\right]
\\
&=\mathrm{Sym}\left[{1\over z_i-1}{1\over t_{L-1}^{(1)}}\left(f_0(t^{(1)})-f_1(t^{(1)})-z_i\sum_{m=2}^{L-1}f_m^{(i)}(t^{(1)})\right){1\over t_{L-1}^{(2)}}f_0(t^{(2)})
\right]. 
\end{align*}

Therefore, the relation \eqref{eq l=0} holds. 
\end{proof}


\bigskip

{\bf Acknowledgements.} 
The author is grateful to 
 T.~Tsuda  
 and Y.~Yamada 
for helpful discussions.
This work was partially supported by Grant-in-Aid for Japan Society for the Promotion Science Fellows 22-2255.


\end{document}